\documentclass[twocolumn,floatfix,reprint]{revtex4-1}
\usepackage{latexsym}
\usepackage{amsmath}
\usepackage{times}
\usepackage{amssymb}
\usepackage{fancyhdr}
\usepackage[T1]{fontenc}
\usepackage[utf8]{inputenc}
\usepackage[english]{babel}
\usepackage{fancyhdr}
\usepackage{url}
\usepackage{hyperref}
\usepackage{color}

\usepackage{graphicx}
\graphicspath{{./figures/}}

\usepackage{algorithm} 
\usepackage{algorithmic}

\usepackage{amsthm}

\def\R{\mathbb R}
\def\C{\mathbb C}

\newcommand{\e}{{\rm e}}
\newcommand{\suml}{\sum\limits}
\newcommand{\intl}{\int\limits}
\newcommand{\cF}{ {\cal F} }

\newcommand{\bone}{ \mathbf{1} }
\newcommand{\ba}{ \mathbf{a} }
\newcommand{\bW}{{\mathbf{W}}}

\newtheorem{thm}{Theorem}
\newtheorem{lemma}{Lemma}

\begin{document}
	

	\title{\bf Parallel numerical method for nonlocal-in-time Schr\"odinger equation}
	\author{Dmytro Sytnyk$^{1}$}\thanks{Email: sytnikd@gmail.com,  Tel.: +38 044 234 55 63}
	\affiliation{$^1$Department of Numerical Mathematics, Institute of Mathematics, National Academy of Sciences, Tereschenkivska 3, Kiev, Ukraine, 01601}
	
	\date{\today}
	
	
	\begin{abstract}
		We present a new parallel numerical method for solving the non-stationary Schr\"odinger equation with linear nonlocal condition and time-dependent potential which does not commute with the stationary part of the Hamiltonian.
		The given problem is discretized in-time using a polynomial-based collocation scheme. We establish the conditions on the existence of solution to the discretized problem, estimate the accuracy of the discretized solution and propose the method how this solution can be approximately found in an efficient parallel manner.
		  
		\noindent 
		
		\vspace*{2ex}\noindent\textit{\bf Keywords}: Schr\"odinger equation, linear nonlocal condition, collocation scheme, existence of solution, iterative approximation, Dunford-Cauchy integral, parallel numerical method.
		\\[3pt]
		\noindent\textit{\bf PACS}:  0.260, 95.10.E, 89.75.-k, 89.75.Da, 05.45.Xt, 87.18.-h 
		\\[3pt]
	\end{abstract}

	\maketitle
	\thispagestyle{fancy}
\section*{Introduction}
	We present a new numerical method for solving the time-dependent Schr\"odinger equation with linear nonlocal condition
\begin{align}
{i}\frac{\partial\Psi}{\partial t}-(H+v(t))\Psi= 0, \label{eq:NCPSchrodEqt}\\
\quad \Psi(0) + \sum\limits_{k=1}^{m} \alpha_k \Psi (t_k)= \Psi_0,
\label{eq:NCPnonloc_cond} 
\end{align}
$\alpha_k \in \mathrm{C}, \ t_k \in (0, T]$, $\Psi_0 \in X$.
It is assumed that $H$ is a densely defined closed linear operator  with the domain $D(H)$ dense in a Banach space $X = X(\|\cdot\|,\Omega)$. 
The spectrum of $H$ is contained in the horizontal half-strip 
\begin{equation}\label{eq:SpHalfStrip}
\Sigma = \left\{
z=x + i y\ \middle|\ x,y \in \R,\ x\geq b_s,\ |y|\leq d_s 
\right\},
\end{equation} 
and the resolvent 
$
R\left(z,H\right) \equiv (zI-H)^{-1}
$ 
satisfies the bound
\begin{equation}\label{eq:ResHalfStrip}
\left \|R\left(z,H\right)\right \|\leq \frac{M}{|\Im{z}|-d_s},\quad z \in \Theta\setminus \Sigma,\  \Sigma \subset \Theta.
\end{equation}
The linear operator $H$ having properties \eqref{eq:SpHalfStrip},\eqref{eq:ResHalfStrip} is called a semi-bounded half-strip operator \cite{batty2012bounded}. 
The class of such operators can be viewed as a native extension of the class of Hermitian operators with a semi-bounded spectrum \cite{Sakurai1985} .
The motivation to consider equation \eqref{eq:NCPSchrodEqt} in a  Banach space setting stems from fact that the technique used in this work does not rely on the notion of inner product. 
Thus the results established here can be readily applied to the conventional quantum mechanical models with Hermitian operators as well as to the less conventional models with $\mathcal{PT}$--symmetric \cite{Bender2002} or  pseudo-Hermitian operators \cite{Mostafazadeh2002}. 
The later type of models is becoming more important due to the recent applications in nonlinear quantum optics  \cite{Chen2017}, \cite{Zyablovsky2014} and 2-D material design \cite{Bagarello2016}. 
Problem \eqref{eq:NCPSchrodEqt}, \eqref{eq:NCPnonloc_cond} has applications in the theory of non-periodic driven quantum systems, quantum computations, and the modelling of system-bath interactions in open quantum systems. The detailed discussion of the above-mentioned applications of \eqref{eq:NCPSchrodEqt}, \eqref{eq:NCPnonloc_cond} are presented in \cite{SytnykMelnik2018}. 


In the current work we consider a general situation when the potential $v(t)$ does not commute with $H$. 
As a consequence of that, the propagator  $e^{-i t (H+v(t))}$ does not commute with itself for different values of $t$.
This issue severely limits the list of analytical and numerical tools applicable to the solution operator $\exp{\left(\int\limits_{0}^t H + v(s) ds \right) }$ of \eqref{eq:NCPSchrodEqt} because  such solution operator is intractable within standard holomorphic function calculus of $H + v(t)$ \cite{batty2012bounded}. 
We refer the reader to \cite{Leforestier1991},\cite{Dijk2011} for a review of available numerical methods to solve Schr\"odinger equation \eqref{eq:NCPSchrodEqt} accompanied by the ordinary initial condition (all $\alpha_k$ from \eqref{eq:NCPnonloc_cond} are zero) and with $H$ being one- or two-dimensional scalar elliptic operator.  
Nonlocal condition \eqref{eq:NCPnonloc_cond} poses an additional issue that contributes to the complexity of the given problem. 
To our best knowledge the only available theoretical work devoted to stationary-operator version ($v(t)=0$) of \eqref{eq:NCPSchrodEqt}, \eqref{eq:NCPnonloc_cond} in its full generality is \cite{SytnykMelnik2018}. 
The particular cases of the given problem was studied in \cite{ashyralyev2008nonlocal}, \cite{Bishop2017}, \cite{bunoiu2016vectorial}. 
Numerical methods for \eqref{eq:NCPSchrodEqt}, \eqref{eq:NCPnonloc_cond} were never reported.
 
To work around the highlighted issues we transfer the time dependent part $v(s)\Psi (s)$ to the right-hand side of \eqref{eq:NCPSchrodEqt} and look for the numerical solution of the obtained problem.
The above assumptions on $H$ guarantee that $e^{-i t H}$ is bounded and any solution to 
\eqref{eq:NCPSchrodEqt} also satisfies the equation 
\begin{equation}\label{eq:NCPPropSolRepSE}
\Psi (t) = e^{-i t H}\Psi(0) +  \int\limits_{0}^{t}e^{-i\left(t-s\right)H}V(s)\Psi (s)ds,  
\end{equation}
with some  $\Psi(0)\in D(H)$, provided that the potential $V(t) \equiv -i v(t)$ is integrable on $[0,T]$ and 
there exists $\delta>1$ such that $D(H^\delta)$ is dense in $X$ (see. \cite[Section 2]{SytnykMelnik2018}). Throughout the paper we 
assume the validity of both these conditions. 

In order to discretize \eqref{eq:NCPPropSolRepSE}, \eqref{eq:NCPnonloc_cond} in-time we propose  in Section \ref{sec:discretization} a polynomial-based collocation scheme on the Chebyshev-Gauss-Lobatto  grid. 
This scheme permits us to reduce nonlocal problem \eqref{eq:NCPPropSolRepSE}, \eqref{eq:NCPnonloc_cond} to a system of linear integral equations. 
Next, we study a well-possedness of the obtained system (see Lemma \ref{thm:S_inv}, \ref{thm:matrix_est}).  
This is done using the combination of previously obtained results \cite{SytnykMelnik2018} together with some specific transformations tailored to the structure of nonlocal condition \eqref{eq:NCPnonloc_cond_transf}. 
Theorem \ref{thm:main} comprises the main result of the work. It states the conditions on the existence of solution to the discretized system and justifies the iterative method to approximate this solution. 


In Section \ref{sec:op_func_approx} we illustrate how the action of propagator $e^{-i t H}\phi$ can be efficiently approximated by the parallel numerical method proposed in \cite{schrod_num_Sytnyk2017}. 
This method reduces the sought approximation to a series of independent stationary problems
$$
(z_kI-H)\Phi = \phi, \quad z_k \in \C.  
$$
that can be solved in parallel. 
Section \ref{sec:impl_ex} is devoted to implementation of the numerical method discussed in the previous sections.
In this section we present the approximation algorithm and discuss its sequential and parallel complexities.   

\section{Discretization scheme}\label{sec:discretization}
To build a discretization scheme we perform the change of variable 
\begin{equation}\label{eq:var_transform}
t = \frac{s+1}{2}T
\end{equation}
in \eqref{eq:NCPSchrodEqt}, \eqref{eq:NCPnonloc_cond} and reduce the given problem on $t \in [0,T]$ to the equivalent problem on $s \in [-1, 1]$
\begin{align}
{i}\frac{\partial\psi}{\partial s}-H\psi = iV(s)\psi, \label{eq:NCP_transf}\\
\quad \psi(-1) + \sum\limits_{k=1}^{m} \alpha_k \psi (s_k) = \Psi_0.
\label{eq:NCPnonloc_cond_transf} 
\end{align}
The sequence of pairs $(\alpha_k, s_k)$, $\alpha_k \in \mathrm{C}, \ s_k \in (-1, 1]$, $k=1,\ldots m$ will be called  parameters of nonlocal condition.

In order to discretize the solution  to \eqref{eq:NCP_transf}, \eqref{eq:NCPnonloc_cond_transf} in-time  we introduce the Chebyshev-Gauss-Lobatto (CGL) grid 
\[
\omega_{N}=\{ s_{p}=-\cos{\frac{p\pi}{N}}, p=0,...,N \}.
\]
It is well-known \cite{pol_int_err_est_Mastroianni2008} that the nodes $s_{p} \in \omega_{N}$ are the zeros of $(1-x^2)T_N^{\prime}(x)$, where $T_N(s)=\cos{(N \arccos{s})}$ is the Chebyshev orthogonal polynomial of the  first kind.
Moreover the step-sizes $\tau_{p} \equiv s_{p}-s_{p-1}$ satisfy the inequality 
\cite[Thm. 6.11.12]{szegoe}
\begin{equation}\label{eq:tau_max}
\tau_{\text{max}}=\max_{1 \le p \le N} \tau_p <  \frac{\pi}{N}
\end{equation}
We seek the solution to \eqref{eq:NCP_transf}, \eqref{eq:NCPnonloc_cond_transf} in the form of polynomial 
\begin{equation}\label{eq:lagrange_pol}
P_{N}(s; \varphi)=\sum_{p=0}^N \varphi(s_p) L_{p}(s),
\end{equation}
where $L_{p}$, 
$p=0,...,N$
are Lagrange fundamental polynomials associated with the grid $\omega_{N}$ and $\varphi: [-1, 1] \rightarrow X$ is some unknown function.

Upon substituting \eqref{eq:lagrange_pol} into \eqref{eq:NCP_transf}
and evaluating the result on the grid $\omega_{N}$ with help of \eqref{eq:NCPPropSolRepSE}, we arrive at the following sequence of equations 
\begin{multline*}
\varphi(s_p) =  e^{-i (s_p + 1) H}P_{N}(-1;\varphi)  \\ 
+ \int\limits_{-1}^{s_p}e^{-i\left(s_p-t\right)H}V(t)P_{N}(t;\varphi)dt. 
\end{multline*}
For any $1 \leq p \leq N$ the previous equation can be rewritten as follows 
\begin{equation}
\begin{split}
\varphi(s_p) &=  e^{-i \tau_p H}\varphi(s_{p-1}) \\
 +  &\sum_{l=0}^N \int\limits_{s_{p-1}}^{s_p} e^{-i\left(s_p-t\right)H}V(t)L_{l}(t)\varphi(s_l) dt .
\end{split}
\label{eq:int_repr_seq_pol}
\end{equation}
To get \eqref{eq:int_repr_seq_pol} we used the interpolation property $P_{N}(s_p;\varphi) = \varphi(s_p)$, $p=1,\ldots,N,$ along with the fact that $H$ does not depend on time, so $e^{-i s_p H} = e^{-i s_{p-1} H}e^{-i \tau_p H}$.
Similarly, the substitution of $P_{N}(s;\varphi)$ into \eqref{eq:NCPnonloc_cond_transf} yields
\begin{equation}\label{eq:nc_transf_pol}
\varphi(s_0)+ \sum_{l=0}^N \sum\limits_{k=1}^{m} \alpha_k  L_{l}(s_k) \varphi(s_l) = \Psi_0.
\end{equation}
Equations \eqref{eq:int_repr_seq_pol}, $p = 1, \ldots, N$ and \eqref{eq:nc_transf_pol} together form a system of $N+1$  linear operator equations with respect to the unknowns $\Phi = (\varphi(s_0), \ldots, \varphi(s_N))$. 
We rewrite this system in a matrix-vector form 
\begin{equation}\label{eq:matrix_sys}
S\Phi = C \Phi + F,
\end{equation}
where
\begin{equation*}
S=
\begin{pmatrix}
 (1+ a_0)I& a_1I & a_2I & \cdot & \cdot & \cdot & a_{N-1}I & a_NI \\
-e^{-iH \tau_1} & I & 0 & \cdot & \cdot & \cdot & 0 & 0 \\
0 & -e^{-iH \tau_2} & I & \cdot & \cdot & \cdot & 0 & 0 \\
\cdot & \cdot & \cdot & \cdot & \cdot & \cdot & \cdot & \cdot \\
0 & 0 & 0 & \cdot & \cdot & \cdot & -e^{-iH \tau_{N}} & I \
\end{pmatrix},
\end{equation*}
$a_l =  \sum\limits_{k=1}^{m} \alpha_k  L_{l}(s_k)$, $C=\{\beta_{p,l}\}_{p,l=0}^N$ is the matrix with entries 
$\beta_{p,l} = \int\limits_{s_{p-1}}^{s_p} e^{-i\left(s_p-t\right)H}V(t)L_{l}(t)dt$,  $\beta_{0,l} = 0$ and $F$ is $N+1$-dimensional vector $F=(\Psi_0,0,\ldots,0)^T$.
The elements of the first row of $S$ are collected from the terms on the left of \eqref{eq:nc_transf_pol}. Other nonzero elements of $S$ come from the first two terms of \eqref{eq:int_repr_seq_pol}, when $p$ goes from  $1$ to $N$.
 
We would like to show that the solution of \eqref{eq:matrix_sys} exists for a sufficiently large $N$ and  then characterize the accuracy of that solution. 
To do so, let us introduce a vector norm 
\begin{equation}\label{eq:vec_norm_X}
|\|v\|| \equiv |\|v\||_\infty=\max_{1 \le k \le n} \|v_k\|
\end{equation}
and the corresponding matrix norm
\begin{equation}\label{eq:mat_norm_X}
|\|A\|| \equiv |\|A\||_\infty =\max_{1 \le i \le n}\sum_{j=1}^n
\|a_{i,j}\|.
\end{equation}
\begin{lemma}\label{thm:S_inv}
Assume that the operator function 
$B_N = I + \sum\limits_{l=0}^{N} a_l e^{-iH (s_l -s_0)}$
possesses a bounded inverse $B_N^{-1}$,
then the matrix $S$ is invertible and the inverse $S^{-1}$ has the following representation 
\begin{equation}\label{eq:S_inv_repr}
S^{-1} = S_1^{-1}\left (I - \bone \ba^T S_1^{-1}B_N^{-1} \right ),
\end{equation}
where 
\begin{equation}\label{eq:S1_inv}
{
\begin{split}
S_1^{-1}=
\begin{pmatrix}
I & 0&   \cdots & 0 & 0 \\
e^{-iH\tau_1} &I  &  \cdots & 0 & 0 \\
e^{-iH (s_2-s_0)} &e^{-iH \tau_2}  &  \cdots & 0 & 0 \\
\cdot & \cdot &  \cdots & \cdot & \cdot \\
e^{-iH (s_N-s_0)} & e^{-iH (s_N - s_1)} & \cdots & e^{-iH\tau_N} & I
\end{pmatrix},
\end{split}
}
\end{equation}
and $\bone = (1,0,\ldots, 0)^T$, $\ba = (a_0,\ldots, a_N)^T$ are two vectors of the same size.	
\end{lemma}
\begin{proof}
To prove \eqref{eq:S_inv_repr} we notice that the matrix $S$ can be decomposed as  $S =  S_1 + \bone \ba^T $, where $S_1$ is a lower bidiagonal matrix with identity operators on the main diagonal. The matrix $\bone \ba^T$ is a rank-1 update of $S_1$. Due to its specific structure, the matrix $S_1$ is always invertible. The inverse $S_1^{-1}$ is defined by \eqref{eq:S1_inv}. Consequently, the inverse $S^{-1}$ exists and can be evaluated via the Sherman-Morrison formula \cite{Hager1989}.
It leads us to the representation
\[
S^{-1} =  S_1^{-1}\left (I - \bone \ba^T S_1^{-1}\left(I+ \ba^T S_1^{-1} \bone \right)^{-1} \right ),
\]
which defines a bounded inverse of $S$, if and only if the operator function $\left(I+ \ba^T S_1^{-1} \bone \right)^{-1}$ is bounded for the given $H$. 
By a direct calculation we get 
\[
\left(I+ \ba^T S_1^{-1} \bone \right) = I + \sum\limits_{l=0}^{N} a_l e^{-iH (s_l -s_0)} \equiv B_N.
\]
\end{proof}	
 
To understand how the function $B_N$ is related to nonlocal condition \eqref{eq:NCPnonloc_cond_transf} we need to recall some results from \cite{SytnykMelnik2018}. In the mentioned work authors studied the problem comprised of
\begin{equation} \label{eq:SchrodEqt}
{i}\frac{\partial\Psi}{\partial t}-H\Psi = iV(t), \quad t \in (0,T]
\end{equation}
and the nonlocal condition \eqref{eq:NCPnonloc_cond}, under slightly more general assumptions on $H$ than in the current work. 
 The existence and representation of solution to \eqref{eq:SchrodEqt}, \eqref{eq:NCPnonloc_cond} relies upon the boundedness of 
 \[B^{-1} = \left ( I +\suml_{k=1}^m \alpha_k e^{-i t_k H}\right)^{-1}.\]
 \begin{thm}[\cite{SytnykMelnik2018}]\label{thm:NCNS_exist_mild}
 	Let  $H$ be a closed linear operator with the spectrum $\Sigma$ contained in strip \eqref{eq:SpHalfStrip} and the domain $D(H^\delta)$ is dense in $X$ for some $\delta > 1$. 
 	The mild solution of nonlocal problem \eqref{eq:SchrodEqt}, \eqref{eq:NCPnonloc_cond} exists for any 
 	$\Psi_0 \in X$, $V \in L^1((0;T),X)$  
 	and is equivalent to the solution of Cauchy problem for \eqref{eq:SchrodEqt}, with the initial state 
 	\begin{equation}\label{eq:initial_state_NCNS}
 	\Psi(0) = B^{-1}\Psi_0 -B^{-1}\suml_{k=1}^m \alpha_k \intl_0^{t_k} e^{-i (t_k-s) H} V(s) ds,
 	\end{equation}
 	if all the zeros of entire function $b(z)$ associated with \eqref{eq:NCPnonloc_cond},
 	\begin{equation}
 	\label{eq:zerosExpI}
 	b(z)=1+\sum_{k=1}^m{\alpha_k e^{(-i t_k z)}}, 
 	\end{equation}
 	are contained in the interior of the set $\mathbb{C} \backslash \Sigma$.
 \end{thm}
 We note that the entire function $b(z)$, describing the existence of the solution in terms of the parameters of nonlocal condition \eqref{eq:NCPnonloc_cond}, is connected to $B^{-1}$ via the Dunford-Cauchy integral 
 \begin{equation}
 \label{reprDunford}
 B^{-1} =\frac{1}{2\pi i} \intl_{\Gamma_I}\frac{1}{b(z)} R(z,H) dz,
 \end{equation}
where $R(z,H)$ is the resolvent of $H$, defined above.
 Hence, the operator function $B^{-1}$ is properly defined and bounded only if the conditions of Theorem \ref{thm:NCNS_exist_mild} regarding the zeros of $b(z)$ are fulfilled.  
 
Now let us get back to the definition of $B_N$. It's not hard to see that $B_N$ is the polynomial  approximation of $B$, transformed under \eqref{eq:var_transform}. This approximation converges quickly as $N$ increases, because $B$ admits holomorphic extension as a function of $s\in [-1,1]$ into the bounded set containing the interval $[-1,1]$ \cite{pol_int_err_est_Mastroianni2008}. Thus, for a sufficiently large $N$, the operator function $B_N^{-1}$ should be bounded when the conditions of Theorem \ref{thm:NCNS_exist_mild} are satisfied.  

\begin{lemma}\label{thm:matrix_est}
Suppose that the potential $V(s)$ from \eqref{eq:NCP_transf} is Lipschitz continuous 
\begin{equation}\label{eq:V_Lip}
\|V(t) - V(s)\| \leq K |t-s|, \quad \forall t,s \in [-1,1],
\end{equation} 
and $M^V=\max_{s \in [-1,1]}\|V(s)\|$, then for a large $N$ the matrices $S^{-1}, C$ and $S^{-1}C$ obey the bounds 
\begin{equation}\label{eq:S_est}
|\|S^{-1}\|| \le
M_S (N+1),
\end{equation}
\begin{equation}\label{eq:C_est}
|\|C\||\le   
\frac{M_{C}}{N+1}\left ( \frac{1}{2} M^V K^L  + \pi K \frac{\ln{(N+1)}}{N+1}\right ),
\end{equation}
\begin{equation}\label{eq:SC_est}
|\|S^{-1}C\|| \le
M_{SC}\left ( \frac{1}{2} M^V K^L  + \pi K \frac{\ln{(N+1)}}{N+1}\right ),
\end{equation}
where the positive constants $M_S$, $M_C$, $M_{SC}$  are independent of $N$,  $K^L$ is the maximum of Lipschitz constants for $L_l(t)$, $t \in [s_{p-1}, s_p]$, $l=0, \ldots, N$.
\end{lemma}
\begin{proof}
Representation  \eqref{eq:S_inv_repr} from Lemma \ref{thm:S_inv} permits us to evaluate $|\|S^{-1}\||$ explicitly 
\begin{multline*}
|\|S^{-1}\|| \leq \left\| B_N^{-1}\right \| \max\limits_{1\leq k \leq N+1} 
\left\{
\vphantom{\sum\limits_{l=0}^{k-1}} \right. \\
\sum\limits_{l=0}^{k-1} 
\left \|
e^{-i(s_{k-1}-s_l)H }
\left(I+\sum\limits_{j=0}^{l-1}a_j e^{-i (s_j-s_0) H}\right) 
\right \| 
\\
\left. +
\left \|
\sum\limits_{l=k}^{N} 
e^{-i(s_{k-1}-s_0)H } 
\sum\limits_{j=l}^{N}a_j e^{-i (s_j-s_l) H}
\right \| 
\right\}
\end{multline*}
Each of $N+1$ terms inside the curly brackets of the above formula contains the product of a bounded propagator term and a part of the sum comprising $B_N$ (see Lemma \ref{thm:S_inv}). The bounded norm of this part is balanced out by a norm of the inverse $\left\| B_N^{-1}\right \|$. Thus, starting from some value of $N$, when $B_N^{-1}$ becomes bounded and close to $B^{-1}$, the ratio of the two norms must be bounded and no longer dependent on $N$.   
Inequality \eqref{eq:S_est} is proved. 

To derive bound \eqref{eq:C_est}, we estimate $\|\beta_{p,l}\|$:
\begin{equation*}
\begin{split}
& \|\beta_{p,l}\| = \left \|\int\limits_{s_{p-1}}^{s_p} e^{-i\left(s_p-t\right)H}V(t)L_{l}(t)dt \right \|\\
&\leq \max_{s \in [0,\tau_p]} \left\|e^{-is H}\right\| \int\limits_{s_{p-1}}^{s_p}\| V(t)L_{l}(t)\|d t.\\
\end{split}
\end{equation*}
Note that $L_{l}(t)$ is zero at least at one endpoint of the interval $t \in (s_{p-1}, s_p)$, $p=0,\ldots, N$.
We pick a smallest of such endpoints and label it as $\theta$, so  $L_{l}(\theta) = 0$.
Then we can add the term $-V(\theta)L_l(\theta)$ to the above integrand without changing the value of the norm inside the integral.
This procedure leads us to the following estimates
\[
\begin{split}
&\int\limits_{s_{p-1}}^{s_p}\| V(t)L_{l}(t)\|d t = 
\int\limits_{s_{p-1}}^{s_p}\| V(t)L_{l}(t) - V(\theta)L_l(\theta)\|d t\\
&\leq \int\limits_{s_{p-1}}^{s_p}|L_{l}(t)|\| V(t) - V(\theta)\| + \|V(\theta)\||L_l(t) - L_l(\theta)|d t\\
&\leq K \tau_{\text{max}} \int\limits_{s_{p-1}}^{s_p}|L_{l}(t)|dt + M^V_p K^L_p \frac{\tau_{\text{max}}^2}{2}. 
\end{split}
\]
To get the last inequality we relied on the Lipschitz continuity of  $V(t)$, expressed by \eqref{eq:V_Lip}, and the fact that the monomials $L_{l}(t)$ are also Lipschitz continuous by definition.
Here $M^V_p$ and $K^L_p$  are the upper bound on $V(t)$ and the Lipschitz constant of $L_l(t)$ on $t \in (s_{p-1}, s_p)$, accordingly. 
%
%
The previous inequality permits us to estimate the norm of $|\|C\||$:

\begin{equation*}
\begin{split}
|\|C\|| \leq &\max_{0 \le p \le N}\sum_{l=0}^N \|\beta_{p,l}\| \\
\leq & \tau_{\text{max}} \max_{s \in [0,\tau_{\text{max}}]} \left\|e^{-is H}\right\| 
\max_{0 \le p \le N}\left ( \frac{1}{2} M^V_p K^L_p 
\right .\\
& \left .
 + K \int\limits_{s_{p-1}}^{s_p}\sum_{l=0}^N|L_{l}(t)|dt  \right )\\
\leq & \tau_{\text{max}} \max_{s \in [0,\tau_{\text{max}}]} \left\|e^{-is H}\right\| 
\left ( \frac{1}{2} M^V K^L  + K \tau_{\text{max}} \Lambda_{N+1}\right ).\\
\end{split}
\end{equation*}
This newly obtained estimate together with \eqref{eq:tau_max} and \eqref{eq:S_est} imply \eqref{eq:C_est}, \eqref{eq:SC_est}.
\end{proof}

Let $\Pi_{N}$ be a set of all polynomials in $s$ of degree less then or equal to $N$ with the coefficients from $X$. Then, the
Lebesgue inequality
\begin{equation}\label{eq:best_approx}
\begin{split}
\max_{s \in	[-1,1]}\| \phi(s)-P_{N}(s; \phi)\| \le (1+\Lambda_{N+1})E_{N}(\phi)
\end{split}
\end{equation}
characterizes the error of the best
approximation of $\phi$ by the polynomials of degree not greater than
$N$,
\begin{equation}\label{eq:EN}
E_N( \phi)=\inf_{P \in \Pi_{N}}\max_{s \in [-1,1]}\|
\phi(s)-P(s)\|.
\end{equation}
Now, we are ready to formulate the main result.
\begin{thm}\label{thm:main}
	Suppose that the assumptions of Theorem \ref{thm:NCNS_exist_mild} are valid. 
	If, for a given $\Psi_0$ and some Lipschitz continuous and bounded $V(s)$, the solution $\psi$ to \eqref{eq:NCP_transf}, \eqref{eq:NCPnonloc_cond_transf} exist, then for a sufficiently large $N$ two following  propositions remain true. 
		\begin{enumerate}
		\item The equation \eqref{eq:matrix_sys} posses a unique solution, which can be found by a fixed point iteration  
		\begin{equation} \label{eq:iter_solution}
		\Phi^{(n+1)} = S^{-1}C \Phi^{(n)} + S^{-1}F, \quad \Phi^{(0)} = \mathbf{0},
		\end{equation}
		provided that the Lipschitz constant $K^\psi$ of $\psi$ satisfies the inequality $\frac{1}{2} M_{SC}M^V K^\psi < 1$, with the quantities  $M_{SC}, M^V$ defined by Lemma \ref{thm:matrix_est}.
		\item The accuracy of solution $\Phi$ to \eqref{eq:matrix_sys} is characterized by the bound: 
		\begin{equation}\label{eq:matrix_sys_bound}
		\|\Psi-\Phi\||\leq M\ln{(N+1)}E_{N}(\phi),
		\end{equation}
		where $\Psi$ is a projection of $\psi$ on $\omega_{N}$ and $M$ is some constant independent of $N$.
	\end{enumerate}
\end{thm}
\begin{proof}
	First of all  we observe that every solution to \eqref{eq:NCPSchrodEqt}, \eqref{eq:NCPnonloc_cond} is also a solution to \eqref{eq:SchrodEqt}, \eqref{eq:NCPnonloc_cond} with $iV(t)\Psi(t)$ in place of $v(t)$. 
	Consequently, there exist some  $\Psi_0$ that corresponds to such solution of \eqref{eq:SchrodEqt}, \eqref{eq:NCPnonloc_cond}. This entails the validity of the statement from Theorem \ref{thm:NCNS_exist_mild} regarding the zeros of $b(z)$ \eqref{eq:zerosExpI}, which, in turn, guaranties that $B^{-1}$ is bounded. 
	As we already mentioned $B_N^{-1} \rightarrow B^{-1}$ ($N \rightarrow \infty$). 
	Thus, we can take $N=N_0$ large enough so that both Lemmas \ref{thm:S_inv}, \ref{thm:matrix_est} are true simultaneously. 
	Then we find $N'>N_0$ from the inequality 
	\[
	M_{SC}\left ( \frac{1}{2} M^V K_n  + \pi K \frac{\ln{(N+1)}}{N+1}\right )<1.
	\]
	The constant $K_n$ here is zero initially, because the initial iteration is zero. 
	When the iteration scheme progresses this constant goes towards Lipschitz constant $K^\psi$ for the exact solution.  
	For any $N \geq N'$ mapping \eqref{eq:iter_solution} is a contraction, provided that the inequality from the theorem's premise regarding $K^\psi$ is valid.
	The Banach fixed-point theorem \cite{Reed2005} concludes the proof of the first part. 
	Estimate \eqref{eq:matrix_sys_bound} needed to prove the second part, follows immediately from \eqref{eq:SC_est}, \eqref{eq:best_approx}.
	
\end{proof}
We would like to remark that the existence result of Theorem \ref{thm:main} could be made independent of the Lipschitz constant $K^\psi$ of the exact solution $\psi$ by reformulating discretized system \eqref{eq:int_repr_seq_pol} as it was done in \cite{Vasylyk2012} for the abstract parabolic equation. 
This reformulation, however, vastly complicates the evaluation of $S^{-1}$ and makes the proposed numerical approach computationally infeasible.
Our preliminary numerical results indicate that the iterative method defined by \eqref{eq:iter_solution} converges, even for the oscillating potentials.
The method given by \eqref{eq:iter_solution} is not the only possible iterative method of approximating the solution to \eqref{eq:matrix_sys}. 
Since this equation is linear in $\Phi$ other Krylov-subspace-based iterative techniques \cite{Liesen2013} might be more effective than \eqref{eq:iter_solution}. 
This is especially true if $H$ is a large sparse matrix obtained as a result of finite-element (FE), boundary-element (BE) or finite-difference (FD) discretization  of the original partial differential operator.  

In principle the elements of $S$, $C$ from \eqref{eq:matrix_sys} can be approximated by any method capable of solving the Cauchy problem for \eqref{eq:SchrodEqt} numerically, see e. g. \cite{batty2012bounded}.  For a whole scheme to be effective however,  the chosen numerical method needs to be able to reuse the previously obtained solutions of stationary problems while  evaluating the sequence $e^{-i s_p H}$, $\beta_{k,p}$ with $p,k=0,\ldots,N$.   

\section{Numerical method for propagator approximation}\label{sec:op_func_approx}	
In this section we illustrate how to build parallel approximation methods for two types of operator functions needed to evaluate $S^{-1}, C$. Those are 
\begin{equation}\label{eq:prop_int_prop}
\psi_h(s)\equiv e^{-is H}\phi, \quad \psi_{ih}(s) \equiv \int\limits_{s'}^s  e^{-i(s-t) H} v(t) dt,
\end{equation}
where $\phi \in X$ and $s'<s$. 
	By applying the Dunford-Cauchy integral representation to \eqref{eq:prop_int_prop} and interchanging the integration order in the second integral we arrive at
		\begin{equation}\label{ContSolRepSE}
		\begin{split}
		\psi_h(s)=
		\frac{1}{2 \pi i}\intl_{\Gamma_I} \e^{-izs}R(z,H)\phi dz \\
		\psi_{ih}(s)=
		\frac{1}{2 \pi i}\intl_{\Gamma_I}  R(z,H) \int_{s'}^{s} \e^{-iz(s-t)}v(t) dt  dz, 
		\end{split}
		\end{equation}
	The function 
	$\psi_h(s)$ can be regarded as a solution of the homogeneous problem for \eqref{eq:SchrodEqt} with the initial condition $\Psi(0) = \phi$. 
	Similarly,  $\psi_{ih}(s)$ is a solution to the inhomogeneous problem for \eqref{eq:SchrodEqt} with the zero initial condition $\Psi(s') = 0$ and $V(s)=v(s)$.
	
	In order to proceed toward the numerical scheme for the approximation of \eqref{ContSolRepSE} we need to define a suitable integration contour $\Gamma_I$.	In doing so we should keep in mind that $\Gamma_I$ must be  positively oriented with respect to the region $\Sigma$ and the integrands need to have a sufficient decay rate for the integrals from \eqref{ContSolRepSE} to converge to \eqref{eq:prop_int_prop}. 
	We choose 
	\begin{equation}\label{eq:int_cont}
	\Gamma_I: z(\xi)= c_I+a_I\sqrt{\frac{\pi}{2}+\xi^2}-i d_I \tanh{\xi},\  \xi \in \R.
	\end{equation}
	The parameters of the contour $\Gamma_I$ are specified as follows 
	\[
	\begin{split}
	a_I &= \frac{d_s}{\pi/2-d},\quad  d_I = \frac{d_s \pi}{\pi -2d}, \\
	c_I &= b_s - a_I\sqrt{\frac{\pi}{2}-d^2} - d_I\tan{d},
	\end{split}
	\]
	where $b_s,d_s$ are defined in \eqref{eq:SpHalfStrip}.
	The parametrization $z(\xi)$ of the contour $\Gamma_I$ defines a conformal mapping of the strip 
	\[
	D_d=\{z \in \C: - \infty < \Re z < \infty, |\Im z|<d \},
	\]
	(see FIG. \ref{fig:env2stripe} b.) into the curvilinear stripe-like region enveloping the half-strip $\Sigma$ (see FIG. \ref{fig:env2stripe} a.).
	
	\onecolumngrid

	\begin{figure}[h!t]
	\centering
	\includegraphics[width=0.4\textwidth]{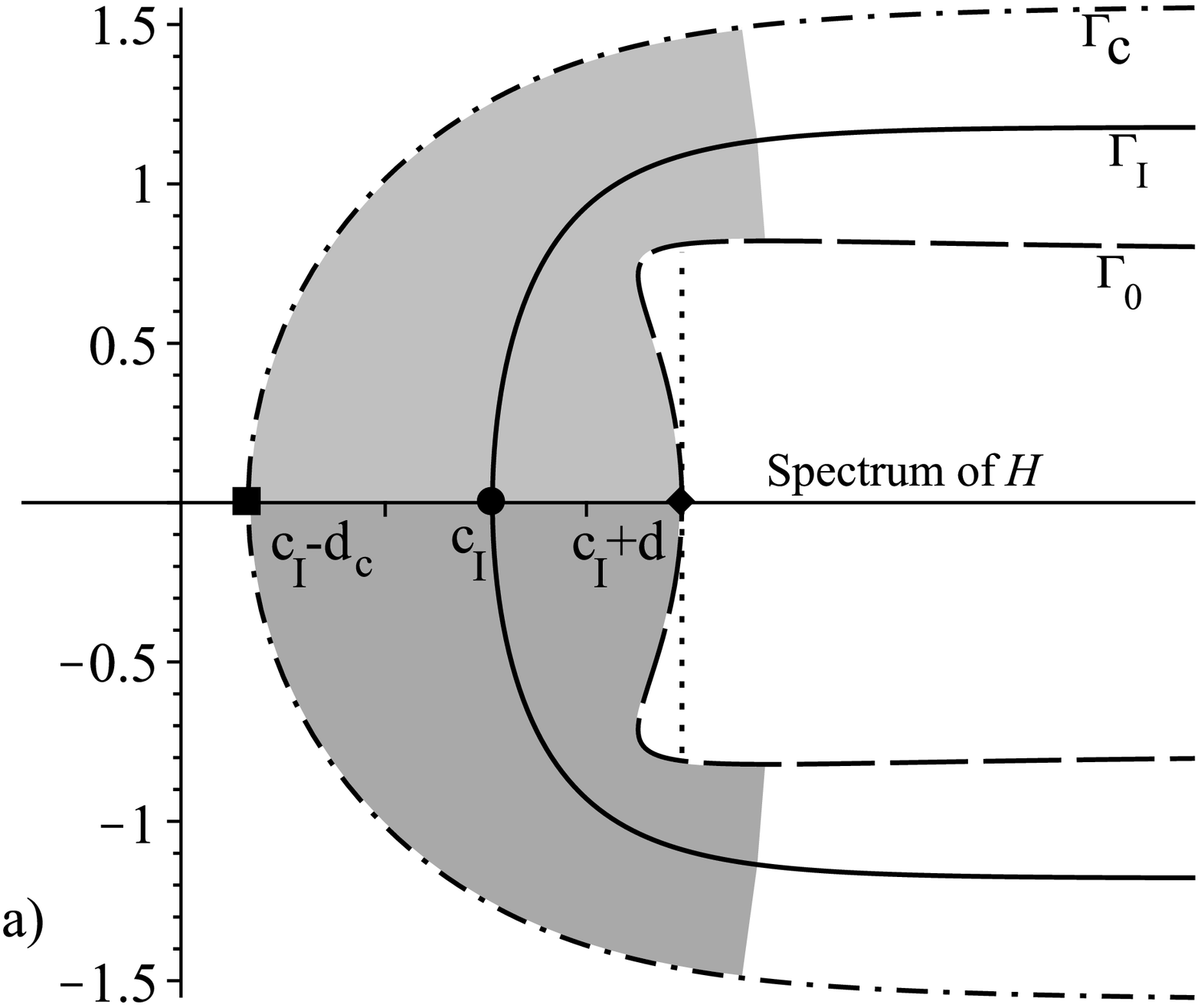}
	\hspace{1.5cm}
	\includegraphics[width=0.4\textwidth]{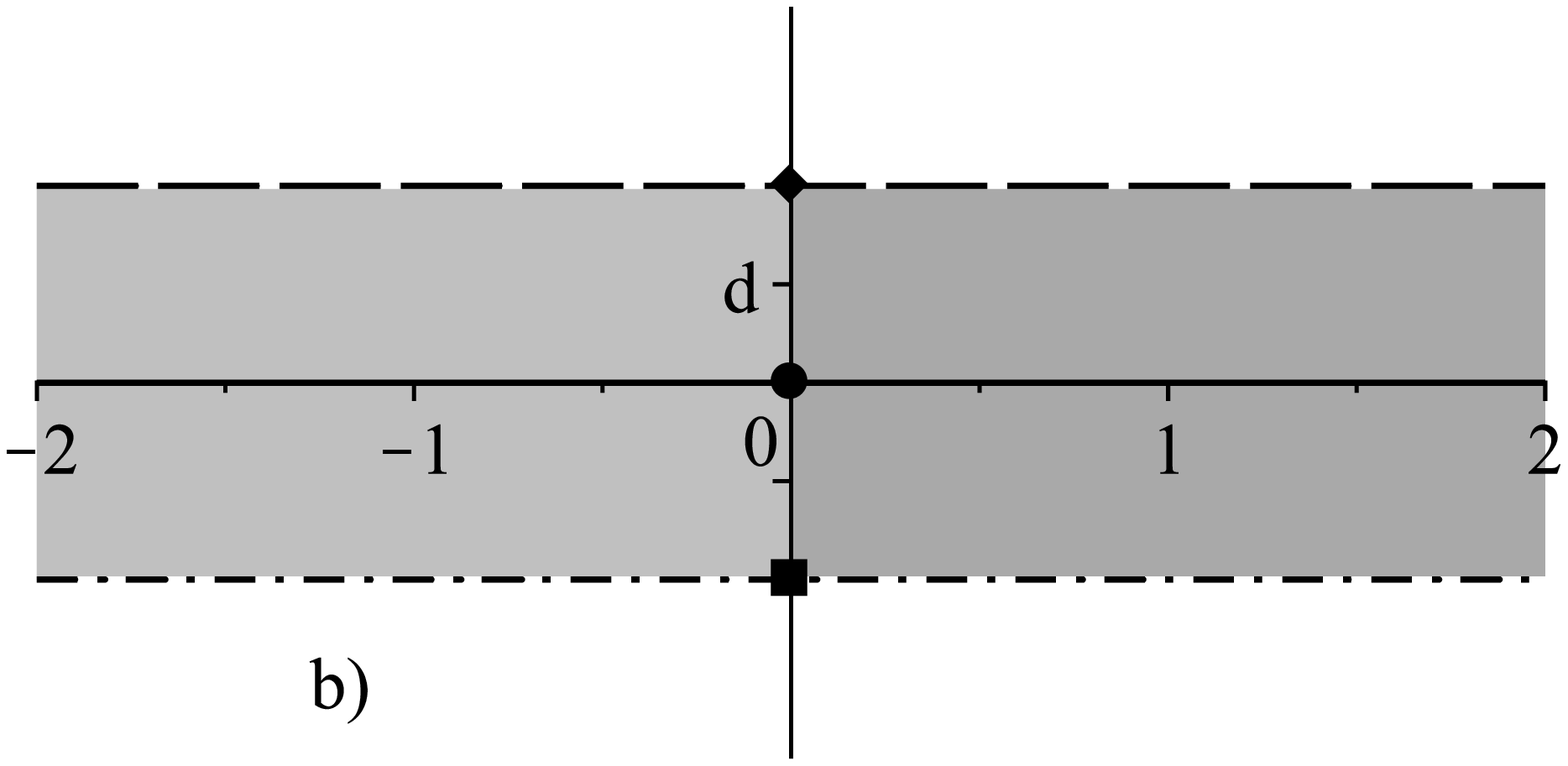} 
	\hfill
	\caption{Contour of integration $\Gamma_I$ ($b_s=\frac{\pi}{2}$, $d_s=\frac{\pi}{4}$, $d_c = d=\frac{\pi}{6}$) and the spectral envelope domain a); Its pre-image infinite horizontal strip b).}
	\label{fig:env2stripe}
	\end{figure}
	\twocolumngrid
	Integrands from \eqref{ContSolRepSE} remain analytic and bounded with respect to $\xi$ for all $\xi \in D_d$. 
	The parameter $0 \leq d \leq \frac{\pi}{6}$ is selected in such a way that all the zeros of $b(z)$ lay outside the mentioned stripe-like region $z(D_d)$.
	
	After parametrization of \eqref{ContSolRepSE} on $\Gamma_I$ we obtain 
	\begin{equation}\label{ParHomSolSE}
	\begin{split}
	\psi_h(s)=& \frac{1}{2 \pi i}\intl_{-\infty}^{\infty}\cF(s,\xi) \phi d \xi, \\
	\psi_{ih}(s)=& \frac{1}{2 \pi i}\intl_{-\infty}^{\infty}\cF(s,\xi) \intl_{s'}^s \e^{i z(\xi)t}v(t) dt d\xi,
		\end{split}
	\end{equation}
	with
	\[
	\begin{split}
	\cF(s,\xi) = &\e^{-iz(\xi)s}F_H(\xi),\\
	\cF_H(\xi) = &z'(\xi) \left[(z(\xi)I-H)^{-1}-\suml_{r=1}^{\lfloor\delta \rfloor}
	\frac{(H-z_0 I)^{r-1}}{(z(\xi)-z_0)^r} \right],\\
	z'(\xi)=&\frac{a_I \xi}{\sqrt{\pi/2+\xi^2}}+i d_I (\tanh{\xi}^2-1).
	\end{split}
	\]
	Here $\lfloor\delta \rfloor$ denotes a floor of $\delta$, i.e. the largest integer number less or equal to $\delta$. 
	In the formulas above we introduced a correction $\suml_{r=1}^{\lfloor\delta \rfloor}\frac{(H-z_0 I)^{r-1}}{(z-z_0)^r(\xi)}$ to the resolvent $R(z,H)$.  
	As discussed in \cite{GMV-mon}, the correction does not change the value of the integral. 
	It is needed to cancel out the first $\lfloor\delta \rfloor$ terms in the Taylor expansion of $R(z,H)$ around  
	\[
	z_0 = \min\left \{0,b_s - a_I\sqrt{\frac{\pi}{2}-d^2} - 1\right \}.
	\]
	If $\phi \in D(H^\delta)$ the corrected resolvent (the part of $\cF_H(s,\xi)$ inside square brackets) will  decay at least as  $|z|^{-\lfloor\delta\rfloor}$, when $z \in \Gamma_I$ and $|z|$ is large enough \cite{GMV-mon}. 
	To ascertain this property, we estimate the norm of the corrected resolvent on $\Gamma_I$:
	\[
	\begin{aligned}
	&\left \|(zI-H)^{-1}-\suml_{r=1}^{\lfloor\delta \rfloor}\frac{(H-z_0 I)^{r-1}}{(z-z_0)^r}	\right \| \\
	& = \left \|\suml_{r=\lfloor\delta \rfloor+1}^{\infty}\frac{(H-z_0 I)^{r-1}}{(z-z_0)^r} \right \|\\
	& =\left \| \left (\frac{H-z_0 I}{z-z_0}\right )^{\lfloor\delta \rfloor}\kern-1em (zI-H)^{-1}\right \|\\
	&\leq |z-z_0|^{-\lfloor\delta \rfloor}\frac{M}{|\Im{z}|-d_s} \left \|(H-z_0 I)^{\lfloor\delta \rfloor}\right \| 
	\end{aligned}
	\] 
	We applied \eqref{eq:ResHalfStrip} to get the above formula. Its last term $(H-z_0 I)^{\lfloor\delta \rfloor}\phi$ is bounded when $\phi \in D(H^\delta)$.
	
	The next auxiliary result describes the accuracy of the trapezoid quadrature rule for the improper integrals similar to \eqref{ParHomSolSE}.
	\begin{thm}[\cite{schrod_num_Sytnyk2017}]\label{thm:SQ_alg_decay_strip}
		Assume that the function $f(z):\C \rightarrow X$ is analytic in the horizontal strip $D_d$, $d>0$. 
		If, for all $z \in D_d$,
		\begin{equation}\label{eq:Alg_dec_strip}
			\|f(z)\| \leq  \frac{L}{1 +  |z|^\delta}, 
		\end{equation}
		with some $\delta > 1$, $L > 0$,
		then the error of trapezoid quadrature rule 
		satisfies the following estimate 
		\begin{equation}\label{eq:err_est_no_N1}
		\begin{split}
		&\left \|\int\limits_{-\infty}^{\infty}f(x)dx - h\sum\limits_{k=-n}^{n}f(kh) \right\| \leq c\frac{(n+1)^{1-\delta}}{(\delta-1)}
		h^{1-\delta},
		\end{split}
		\end{equation}
		provided that 
		\begin{equation*}\label{eq:h_no_N1}
		h=
		\frac{2\pi d }{\delta-1}\left(
		\bW \left(
		\frac{2\pi d }{\delta-1}
		\left(
		\frac{\beta (\delta - 1)}
		{\pi d}
		\right)^{\frac{1}{\delta-1}} 
		(n+1)
		\right)
		\right)^{-1},
		\end{equation*}
		with $\beta = \min\left\{
		\tfrac{2\pi \delta^{-1}}{\sin{\left ( \pi\delta^{-1} \right )}},
		\left(\frac{2}{d}\right)^{\delta - 1} B\left(\frac{\delta}{2}-\frac{1}{2}, \frac{\delta}{2}+\frac{1}{2} \right)
		\right\}.
		$
		Here $B(\cdot,\cdot)$ is the beta function, $c$ is the constant dependent on $\delta, d, L$ and independent on $n$,  $\bW(\cdot)$ denotes a positive branch of the Lambert-W function \cite{Corless1996}, i.e. for any given  $x>0$,  $\bW(x)$ is a unique positive solution of $\bW e^{\bW} = x$.
	\end{thm}


	We assume that $\phi \in D(H^\delta)$ with some $\delta>1$ and approximate $\psi_h$ from \eqref{ParHomSolSE} by the following formula  
	\begin{equation}\label{eq:ApproxHomSolSE}
	\psi_h \approx \psi_{h,n}(s)=\frac{h}{2 \pi i}\suml_{j=-n}^{n}\cF(s,jh)\phi,
	\end{equation}
	where $h$ is specified by Theorem \ref{thm:SQ_alg_decay_strip}.
	Similarly, for the term $\psi_{ih}$ we use the same trapezoid quadrature rule for the outer integral:
	\begin{equation}\label{eq:ApproxInhomSolSE}
	 \begin{split}
	\psi_{ih}(s)\approx \psi_{ih,n}(s)=\frac{h_1}{2 \pi
		i}\sum_{j=-n}^{n}\cF(s,jh) \mu_{j}(s) ds.
	 \end{split}
	\end{equation}
	The inner integral $\mu_{j}(s) = \intl_{s'}^s \e^{iz(\xi)t}\phi dt$ does not depend on $H$, and hence can be approximated  directly.
	The numerical methods represented by \eqref{eq:ApproxHomSolSE}, \eqref{eq:ApproxInhomSolSE} reduce the approximation of \eqref{eq:prop_int_prop} to the sequence of resolvent evaluations $R(z(jh),H)$. By definition each resolvent evaluation is equivalent to the solution of the stationary problem 
	\begin{equation}\label{eq:res_eq}
		(z(jh)I-H)\Phi = g,
	\end{equation}	
	where $g=\phi$ in case of \eqref{eq:ApproxHomSolSE}, and $g=\mu_{j}(s)$ in case of \eqref{eq:ApproxInhomSolSE}.
	All those problems are mutually independent, hence can be solved in parallel. 
	
	According to Theorem \ref{thm:SQ_alg_decay_strip} the error of approximation of \eqref{ParHomSolSE} by \eqref{eq:ApproxHomSolSE}, \eqref{eq:ApproxInhomSolSE} is characterized by estimate \eqref{eq:err_est_no_N1} having the convergence rate on the order of $\mathcal{O}((n+1)^{1-\lfloor\delta\rfloor})$ (in the big-O notation). 
	In that regard, the proposed method is on par with other available numerical methods for propagator approximation \cite{Leforestier1991}. 
	The distinctive feature of the current method is that neither contour $\Gamma_I$ nor parameters $h$, $\delta$ are in any way dependant on $s$. 
	After numerical evaluation was performed once for some $s$, the propagator approximation formula \eqref{eq:ApproxHomSolSE} permits us to evaluate $\psi_h(s)$ for any other value of $s$ without re-evaluation of $R(z(jh),H)$. 
	It is possible because in such scenario the sequence of stationary problems \eqref{eq:res_eq} needs to be solved only once.
	
	As we already mentioned, the convergence order of the proposed approximation  is specified by the decay properties of $\|\cF(s,\xi)\phi\|$ as $z \in \Gamma_I$, $z \rightarrow \infty$. 
	The speed of decay, in turn, depends on the boundedness of the factors $H^r\phi, r=0,1, \ldots$. 
	So, if the  element $\phi \in X$ belongs to the domain of $H^\delta$ for some integer $\delta>1$, i.e. all the powers $H^r\phi$, $r\leq\delta$ are bounded, then the approximation will converge with the algebraic order $\delta-1$. For example, when $H$ is a second order partial differential operator, the property $\phi \in D(H^\delta)$ means that the function $\phi$ along with its first $2\delta$ derivatives are bounded in the region $\Omega$ (see the definition of $X$ above)\cite{batty2012bounded},\cite{GMV-mon}. 

	In practice, the upper bound on the value of $\delta$ also depends on the numerical method chosen to solve \eqref{eq:ApproxHomSolSE}, as one needs to be able to accurately evaluate the corrections to the resolvent on the same grid where resolvent equation \eqref{eq:ApproxHomSolSE} is solved.  
	For FE and BE discretization methods, $\delta$ would depend on the order of the FE- or BE- primitive element's shape functions. 
	Similarly for FD approximations, the optimal value of $\delta$ is related to the order and the type of the scheme used for the space discretization of \eqref{eq:ApproxHomSolSE}. 
	The optimal choice of $\delta$ in each specific case deserves a separate study and is therefore omitted here. 
	For this reason, we also omit the discussion on on how to balance the error estimates of methods from sections \ref{sec:discretization} and \ref{sec:op_func_approx}. In the next section we focus on the algorithmic aspects of the compound numerical method.

	\section{Implementation}\label{sec:impl_ex}
	In this section we present an algorithm to solve discretized version \eqref{eq:matrix_sys} of the translated nonlocal problem expressed by \eqref{eq:NCP_transf}, \eqref{eq:NCPnonloc_cond_transf}. 
	The following algorithm is based on the iterative method proposed in Theorem \ref{thm:main}. 
	It uses the methods of Section \ref{sec:op_func_approx} to evaluate the elements of $S^{-1}$ and $S^{-1}C$ from matrix equation \eqref{eq:matrix_sys}. 
	
	To begin with, it is worthwhile to point out that the second term from iterative formula \eqref{eq:iter_solution} can be simplified in the following way
	\begin{align*}
	S^{-1}F &=  S_1^{-1}\left (I - \bone \ba^T S_1^{-1}B_N^{-1} \right ) F  \\
	&=S_1^{-1}F -  S_1^{-1}\bone \ba^T S_1^{-1}B_N^{-1}F \\
	&=  S_1^{-1}F - S_1^{-1}\bone \left( \Psi_0 - B_N^{-1} \Psi_0 \right) \\
	&= \left ( B_N^{-1} \Psi_0,  \ldots, e^{-iH (s_N-s_0)}  B_N^{-1} \Psi_0\right )^T 	
	\end{align*}
	The calculation of $S^{-1}\Upsilon $ for a general vector $\Upsilon = \left ( \Upsilon_0, \ldots, \Upsilon_N \right )^T$ yields $S^{-1}\Upsilon =\left(S^\Upsilon_0, \ldots,  S^\Upsilon_N\right)^T$,
	\begin{equation*}
	\begin{aligned}
		S^\Upsilon_k =&  \suml_{l=0}^{k}e^{-iH(s_k-s_l) } 
		\left(I + \suml_{p=0}^{l-1} a_p e^{-iH(s_{p}-s_0)}\right)\Upsilon^B_l \\
		&- \suml_{l=k+1}^{N}e^{-iH(s_k-s_0)} \suml_{p=l}^{N} a_p e^{-iH(s_{p}-s_l)}\Upsilon^B_l, \\
	\end{aligned}
	\end{equation*}
	where $\Upsilon^B_l = B_N^{-1} \Upsilon_l$. 
	
	Each iteration \eqref{eq:iter_solution} of the numerical method to solve \eqref{eq:matrix_sys} involves the evaluation of product $S^{-1}C \Phi^{(n)}$. 
	The elements of matrix $C$ can be pre-calculated only when the potential $V(s)$ does not depend on the space variable. 
	For such $V(s)$, of course, the propagator of \eqref{eq:NCP_transf} would necessary commute with itself at different times and all the analysis performed in the paper could be greatly simplified. 
	In a general situation one can not pre-calculate $C$ alone because its elements $b_{p,l}$ contain operator functions of $H$ acting on the product $V(s)\Phi^{(n)}$. 
	Let $\Upsilon = C \Phi^{(n)}$ and $\Phi^B_j = B_N^{-1} \Phi_j$, then the $k$-th element of $C^{\Upsilon} \equiv S^{-1}C \Phi^{(n)}$ can be represented as follows
	\[
\begin{aligned}
		C^\Upsilon_k = & \suml_{j=0}^N \left [
		\suml_{l=0}^{k}
		\left(I + \suml_{p=0}^{l-1} a_p e^{-iH(s_{p}-s_0)}\right)e^{iH s_l} \beta_{lj}
		\right . \\
		& \left . - \suml_{l=k+1}^{N} \suml_{p=l}^{N} a_p e^{-iH(s_{p}-s_0)} e^{iHs_l} \beta_{lj}
		 \right ] e^{-iHs_k} \Phi^B_j.    \\
\end{aligned}	
	\]
	After simplification of the above formula we get 
	\begin{equation*}
		C^\Upsilon_k = \suml_{j=0}^N \suml_{l=1}^{N} f(s_k,H,l) \int\limits_{s_{l-1}}^{s_l} e^{iH t}V(t)L_{j}(t)\Phi_jdt,
	\end{equation*}
	where 
	\begin{equation*}\label{eq:SF_el_calc_func}
	f(s,z,l) = \frac{e^{-izs }}{b_N(z)}
	\begin{cases}
	 1 + \suml_{p=0}^{l-1} a_p e^{-iz(s_{p}-s_0)}, &s \geq s_l,\\
	- \suml_{p=l}^{N} a_p e^{-iz(s_{p}-s_0)}, &s < s_l. 
	\end{cases} 
	\end{equation*}
	The function  $b_N(z) = 1 + \sum\limits_{l=0}^{N} a_l e^{-iz (s_l -s_0)}$ is a scalar analogue of the operator $B_N$. 	
	Elements of $S^{-1}F$ can be numerically evaluated using formula \eqref{eq:ApproxHomSolSE} with 
	\[
	\cF(s,\xi) F = f(s-s_0,z(\xi),0) F_H(\xi) \Psi.
	\] 
	For the elements 	$C^\Upsilon_k$ of $S^{-1}C \Phi^{(n)}$ we get
	\begin{align}
		C^\Upsilon_k &= \suml_{j=0}^N g(s_k,H,j)\Phi_j,\label{eq:Sinv_Ck}\\
		g(s,z,j) &= \suml_{l=1}^{N} f(s,z,l) \int\limits_{s_{l-1}}^{s_l} e^{iz t}V(t)L_{j}(t)dt .\label{eq:Sinv_Ck_g}
	\end{align}
	Similarly to $S^{-1}F$, the action of function $g(s,H,j)$ on the element $\Phi_j$ is approximated using formula \eqref{eq:ApproxInhomSolSE} with 
	\[
	\cF(s,\xi) \Phi_j =  F_H(\xi) g(s,z(\xi),j) \Phi_j.
	\]
	
	For the convergence of approximation formulas \eqref{eq:ApproxHomSolSE}, \eqref{eq:ApproxInhomSolSE} it is critical to maintain a separation between the zeros of $b_N(z)$ and the region  $\Sigma$ defined by \eqref{eq:SpHalfStrip}. More precisely, for a chosen $N$ it must be ensured that the zeros of $b_N(z)$ lay outside the strip-like region depicted on FIG. \ref{fig:env2stripe}.
	Theoretically this separation for $N$ large enough is guaranteed by Theorem \ref{thm:main}.
	In order to achieve it practically one needs to choose $N$ so that all the zeros of $b_N(z)$ lay outside the region bounded by $\Gamma_0$ (see FIG. \ref{fig:env2stripe}) 
	\[
		\Gamma_0:  z(\xi)= c_I+a_I\sqrt{\frac{\pi}{2}+\left (\xi+i\frac{\pi}{6}\right )^2}-i d_I \tanh{\left (\xi + i\frac{\pi}{6}\right )},
	\]
	with $c_I, a_I, d_I$ being calculated for the given pair of spectral parameters $b_s,d_s$ and the strip parameter $d = \pi/6$.
	Then, find a critical value $d_{\mathrm{c}}$ by solving the equation 
	\[
		z(\xi - id_{\mathrm{c}}) = z_{\mathrm{c}},
	\]
	where $z_{\mathrm{c}}$ is the zero of $b_N(z)$ closest to the curve $\Gamma_0$.
	After that, perform the following substitution in formula \eqref{eq:int_cont} 
	\[
	\xi = \left (\frac{3}{\pi}d_{\mathrm{c}} +\frac{1}{2} \right )\nu + i \left( \frac{\pi}{12} - d_{\mathrm{c}}\right).
	\]
	This variable transformation 
	makes the admissible part of the strip-like region $z(\xi)$ depicted on FIG \ref{fig:env2stripe} a), where $\Im(\xi) \in (-d_{\mathrm{c}}, \pi/6)$, symmetric with respect to the imaginary part of the new variable $\nu $: $\Im(\nu) \in (-\pi/6, \pi/6)$. 
	As a result, the curve $z(\nu-i\pi/6)$ goes trough $z_{\mathrm{c}}$ and the curve $z(\nu+i\pi/6)$ coincides with $\Gamma_0$. 
%

	Every summand in the representation of $C^\Upsilon_k$ from \eqref{eq:Sinv_Ck} acts upon a different element $\Phi_j \in X$. 
	Consequently the evaluation of different $C^\Upsilon_k$ requires a re-evaluation of $R(z(m h),H)\Phi_j$ for the same sequence of $j=\overline{0,N}$. 
	To optimize the computations, in Algorithm \ref{alg:sol_it_sys} we evaluate the sum in \eqref{eq:Sinv_Ck} term-by-term for all $C^\Upsilon_k$, $k=\overline{0,N}$ at once. 
	This result in a more computationally efficient process because all operator functions dependent on the given set of evaluated values  $R(z(m h),H)\Phi_j$, $m=\overline{-n,n}$  are calculated in a row (streamlined).  
	
	Before starting to discuss algorithmic implementation, we would like to highlight two computationally useful properties of 
	\eqref{eq:Sinv_Ck},\eqref{eq:Sinv_Ck_g}. 
	The integrands in \eqref{eq:Sinv_Ck_g} do not contain the terms dependent on $s_k, \Phi_j$ from \eqref{eq:Sinv_Ck}. 
	Therefore, once the integrals are computed, they can be reused multiple times, while calculating $C^\Upsilon_k$. Moreover, for a large class of rational potentials $V(s)$ the mentioned integral from \eqref{eq:Sinv_Ck_g} can be evaluated analytically\footnote{The integral admits analytical representation when $V(s)$ is a ratio of two polynomials with coefficients dependent on the space variable. This includes polynomial and more generally Pade approximants.}. 

	The following algorithm calculates the action of the operator function, given by $f$, on a vector $\phi$. It will be used as a subroutine in Algorithm \ref{alg:sol_it_sys}.
	\begin{algorithm}[H]
	\caption{Calculate operator function of $H$ using \eqref{eq:ApproxHomSolSE}}
	\label{alg:o_f}
	\begin{algorithmic}[1]
	\STATE{\textbf{function} O\_F($f(t,z,p),\phi,n,\delta$)}
	\STATE{Calculate $h$ with help of Theorem \ref{thm:SQ_alg_decay_strip}}
		\FOR{$m=-n$ \TO $n$}\label{code:res_eqs}
\STATE{Solve $(z(mh)I-H)\Phi_m = \phi$}
\STATE{Apply the correction $\Phi_m := \Phi_m-\suml_{r=1}^{\lfloor\delta \rfloor}\frac{H^{r-1}\phi}{(z(m h)-z_0)^r}$}
\ENDFOR
\FOR{$l=0$ \TO $N$}
\IF{$l<p$} 
\STATE{Set $S_l:=0$}
\ELSE
\STATE{Evaluate $S_l:=\frac{h}{2 \pi i}\suml_{m=-n}^{n}z'(mh){f\left (s_{l},z(mh),p\right )}\Phi_m$}
\ENDIF
\ENDFOR
\RETURN{$S := \left(S_0, \dots , S_N \right)$ }
\end{algorithmic}
\end{algorithm}
	The non-trivial applications of \eqref{eq:NCP_transf} usually involve the solution of resolvent equation \eqref{eq:res_eq} trough a reduction of this equation (typically differential) to the linear system of algebraic equation via the chosen discretization procedure. 
	The number of unknowns in the resulting linear system is, as a rule, much larger than a size $N$ of the in-time grid $\omega_{N}$. 
	In such typical scenario, the computational complexity of Algorithm \ref{alg:o_f} is dominated by the complexity of the first loop, where $2n+1$ resolvent evaluations are performed. We will say that the sequential complexity of such evaluation is $(2n+1)$, assuming that a resolvent equation fully fits into the memory of one processing unit. As we already mentioned above, all resolvent evaluations can be performed in parallel. By that means, the parallel computational complexity can be reduced to $1$ (using $(2n+1)$ processing units), ignoring the communication overhead. 
	Such impressive complexity reduction is possible because there is no data dependency between the different steps of the mentioned loop from Algorithm \ref{alg:o_f}. 

	Now we turn to the main algorithm implementing iterative process \eqref{eq:iter_solution}.
	In the following algorithm we assume that $V(s)$ and the parameters  $\Psi_0$, $\alpha_k, s_k$, $k=1,2,\ldots m$  of nonlocal condition \eqref{eq:NCPnonloc_cond_transf} are given.
	\begin{algorithm}[H]
	\caption{Iterative solution of \eqref{eq:matrix_sys} via \eqref{eq:iter_solution}}
	\label{alg:sol_it_sys}
	\hspace*{\algorithmicindent} \textbf{Input:} $N$, $n$, $err\_tol$, $max\_it$ \\
	\hspace*{\algorithmicindent} \textbf{Output:} Approximate solution of \eqref{eq:matrix_sys}\par
	\begin{algorithmic}[1]
		\STATE{Set $\Phi^{(0)} := \mathbf{0}$; $it :=0$ }
		\STATE{Calculate $C$}		
		\STATE{Set $S^F := \mathrm{O\_F}(f(t,z,0),\Psi_0,n,\delta)$ }
		\REPEAT
			\STATE{Set $\Upsilon := C \Phi^{(it)}$; $S^\Upsilon := \mathbf{0}$}
			\FOR{$k=0$ \TO $N$}
					\STATE{$S^\Upsilon := S^\Upsilon + \mathrm{O\_F}( f(t,z,k),\Upsilon_k,n,\delta)$ }
			\ENDFOR
			\STATE{Set $\Phi^{(it+1)} := S^\Upsilon + S^F$}
			\STATE{Update error $err_{it+1} := |\|\Phi^{(it+1)} - \Phi^{(it)}\||$}
			\STATE{Set $it := it+1$} 
		\UNTIL{$err_{it} \leq err\_tol$ \OR $it>max\_it$}
		\RETURN $\Phi^{(it)}$
	\end{algorithmic}
	\end{algorithm}
	The sequential computational complexity of each iteration of Algorithm \ref{alg:sol_it_sys} is equal to $(2n+1)(N+1)$ resolvent evaluations plus $(2n+1)$ needed to start the iteration process. 
	All resolvent evaluations are again independent. Owing to that the parallel computational complexity of every iteration can be brought down to $1$, under condition that computational system contains at least $(2n+1)(N+1)$ processing units. To conclude this part we recall that in our complexity metric $1$ is a time needed to solve a stationary problem for a given $H$.

	\section*{Conclusions}\label{sec:conclussions}
			
In this work we developed a new collocation-based numerical method for non-stationary Schr\"odinger equation with non-commuting time-dependent Hamiltonian and linear nonlocal condition. 
Under rather general assumptions we established the existence conditions for the solution of semi-discretized version of the given nonlocal problem. 
Furthermore, we proposed and justified the iterative method to approximate the sought solution.
In addition we've shown how each step of the proposed iterative method can be numerically evaluated using the parallel algorithm with the convergence adjusted to the smoothness of initial data of the given problem. 
The implementation details and computation complexity of the proposed numerical method have been also discussed.
Due to its general formulation, the developed method can be applied to a wide variety of time-dependent problems without constraints on dimensionality or the structure of stationary state space. 
The method can also be used in conjunction with space discretization methods to obtain a fully-discrete numerical scheme. 
	\onecolumngrid

	\def\bibpath{./bib}
	\bibliographystyle{apsrev4-1}
	\bibliography{%
	\bibpath/GMV,%
	\bibpath/polynomial_interpolation,%
	\bibpath/nonlocal_in_time,%
	\bibpath/nonlocal_schrodinger,%
	\bibpath/operator_calculus,%
	\bibpath/sytnyk,%
	\bibpath/sinc_refinement,%
	\bibpath/lin_sys,%
	\bibpath/kp%
	}
\end{document}